\documentclass[a4paper,11pt]{article}

\IfFileExists{styleart.sty}{\input{styleart.sty}}{%
  \usepackage{amsmath,amssymb,amsthm}
  \newtheorem{theorem}{Theorem}[section]
  \newtheorem{proposition}[theorem]{Proposition}
  \newtheorem{lemma}[theorem]{Lemma}
  \newtheorem{fact}[theorem]{Fact}
  \theoremstyle{definition}
  \newtheorem{definition}[theorem]{Definition}
}

\title{Uniform small cancellation in torsion-free hyperbolic groups}
\date{\today}

\author{O. Kharlampovich, R. Sklinos}

\begin{document}

\maketitle

\begin{abstract}
We establish a uniform small cancellation theorem for non-elementary
torsion-free hyperbolic groups: given any prescribed small cancellation
constant, finitely many fixed words suffice to produce a small
cancellation element from every pair of noncommuting elements.

As an application, we prove that the first-order theory of every
non-elementary torsion-free hyperbolic group does not admit a model
companion.
\end{abstract}

\medskip
\noindent\textbf{MSC 2020:}
Primary 03C10; Secondary 20E05, 20F65.
\medskip
\section{Introduction}

Small cancellation phenomena play an important role in the study of
equations and first-order properties of groups. For free groups, the
geometry of the Cayley tree makes it possible to construct, uniformly
from arbitrary noncommuting elements, elements satisfying arbitrarily
strong small cancellation conditions. In our previous work on the
nonexistence of a model companion for the theory of nonabelian free
groups, such a uniform construction was a key ingredient. More
precisely, using results of Casals-Ruiz, Garreta, and de la Nuez
Gonz\'alez \cite{MR4206600}, we obtained, for every prescribed small
cancellation constant, finitely many words such that one of them
produces a small cancellation element from any pair of noncommuting
elements.

The main purpose of the present paper is to establish an analogous
uniform small cancellation theorem for torsion-free hyperbolic groups.
The passage from free groups to hyperbolic groups introduces a genuine
geometric difficulty. In a tree, axes are uniquely determined and
their overlap is measured by their actual intersection. In a
hyperbolic Cayley graph, axes and intersections have to be replaced by
coarse counterparts, and the uniform control required for the
construction is correspondingly more delicate. Nevertheless, we show
that the uniform phenomenon persists. Our main result is the following.

\begin{theorem}\label{thm:main-sc}
Let $\Gamma$ be a non-elementary torsion-free hyperbolic group with a
fixed finite generating set. For every positive integer $N$, there exists
a finite set of words
\[
W_N\subseteq \mathbb F(x,y)
\]
such that, for every pair of noncommuting elements $a,b\in\Gamma$, there
exists $w\in W_N$ for which $w(a,b)$ satisfies the $N$-small cancellation
condition.
\end{theorem}

The uniformity is the essential point of the theorem: for a prescribed
$N$, the choice of the small cancellation element can be made from a
fixed finite family of words, independently of the particular
noncommuting pair.

As an application, we return to the model-theoretic question that
motivated the uniform construction in the free group case. Combining
Theorem~\ref{thm:main-sc} with a suitable version of the implicit
function theorem for hyperbolic groups, we obtain the following.

\begin{theorem}\label{thm:model-companion}
Let $\Gamma$ be a non-elementary torsion-free hyperbolic group. Then
$\operatorname{Th}(\Gamma)$ does not admit a model companion.
\end{theorem}

The model-theoretic obstruction is the same one that appears in the
free group case. If a model companion existed, the set of nonsquares
would be existentially definable in it. Uniform small cancellation,
together with the implicit function theorem, allows one to transfer
such an existential condition to a suitable generic element. But an existential formula which holds of a generic element also holds of each of its nontrivial powers. In particular, the formula defining nonsquares would hold of a square, yielding a contradiction. The main new ingredient needed to carry out this argument is the uniform small cancellation theorem for torsion-free hyperbolic groups proved in this paper.

\medskip
\noindent{\bf Acknowledgments.}
The first-named author thanks the Dolciani--Halloran Foundation for its
support. The second-named author was supported by the National Natural
Science Foundation of China (NSFC), grant no.~12350610234.

\section{Preliminaries}

\subsection{Geometry of hyperbolic groups}
We begin by recalling some background on hyperbolic groups and their actions on hyperbolic spaces. Many of the notions and arguments below are coarse analogues of their counterparts for actions on trees, and the tree case will serve as a useful geometric model throughout.

Let $\Gamma$ be a torsion-free hyperbolic group. Fix a finite generating
set and let $X$ be the geometric realization of the Cayley graph, equipped with its path metric in which every edge has length $1$. 
We fix a hyperbolicity constant $\delta\geq 1$. 

For $v\in\Gamma$, write
\[
 \tau(v)=\lim_{n\longrightarrow\infty}\frac{d(x,v^nx)}{n}
\]
for its stable translation length. This limit is independent of $x$.
Associated with the hyperbolic space $X$ is its \emph{Gromov boundary},
denoted by $\partial X$. Informally, $\partial X$ records the asymptotic
directions in which one can escape to infinity in $X$. More precisely,
fix a basepoint $o\in X$ and define the Gromov product by
\[
 (x,y)_o=\frac12\bigl(d(o,x)+d(o,y)-d(x,y)\bigr).
\]
A sequence $(x_n)$ in $X$ is said to converge to infinity if
\[
 (x_n,x_m)_o\longrightarrow\infty
 \qquad\text{as }n,m\longrightarrow\infty.
\]
Two such sequences $(x_n)$ and $(y_n)$ are equivalent if
\[
 (x_n,y_m)_o\longrightarrow\infty
 \qquad\text{as }n,m\longrightarrow\infty.
\]
The Gromov boundary $\partial X$ is the set of equivalence classes of
such sequences. This definition is independent of the choice of the
basepoint $o$. Equivalently, since $X$ is a proper geodesic hyperbolic
space, $\partial X$ may be identified with the set of geodesic rays
starting at $o$, where two rays are identified if they remain at bounded
Hausdorff distance.

If $v\neq1$, let $v^-,v^+\in\partial X$ be its repelling and attracting
fixed points, and let $A(v)$, which we call the quasi-axis of $v$ be the union of all bi-infinite geodesics joining them. Any two bi-infinite geodesics joining $v^-$ to $v^+$ are at uniformly
bounded Hausdorff distance. Consequently, $A(v)$ is uniformly
quasiconvex: there exists $Q$, depending only on $X$, such that, for
all $x,y\in A(v)$,
\[
[x,y]\subseteq A(v)^{+Q}.
\]
We orient each geodesic in $A(v)$ from $v^-$ towards $v^+$. We write
\[
 C_\Gamma(v)=\{r\in\Gamma:[r,v]=1\}
\]
for the centralizer of $v$. Since $\Gamma$ is torsion-free and
hyperbolic, $C_\Gamma(v)$ is the unique maximal cyclic subgroup
containing $v$. It is malnormal and equal to its normalizer. It also
coincides with the stabilizer of $\{v^-,v^+\}$: this stabilizer is
virtually cyclic, hence cyclic because $\Gamma$ is torsion-free.
In particular,
\begin{equation}\label{Malnormal}
 C_\Gamma(v)\cap rC_\Gamma(v)r^{-1}\neq\{1\}
 \quad\Longrightarrow\quad r\in C_\Gamma(v).
\end{equation}

In a tree, the overlap of two axes is measured by the length of their
intersection. In the hyperbolic setting, where geodesics are only
coarsely determined, we replace this by a coarse notion of overlap. For a subset $Y\subseteq X$, denote its closed $r$-neighbourhood by
$Y^{+r}$. For nontrivial $v,w\in\Gamma$, put
\[
 \Delta(v,w)=\operatorname{diam}
 \bigl(A(v)^{+100\delta}\cap A(w)^{+100\delta}\bigr),
\]
where the diameter of the empty set is $0$. Thus
$A(fvf^{-1})=fA(v)$.

Enlarging $\delta$ once and for all, we assume that
\begin{equation}\label{FourPoint}
 (x,z)_y\geq\min\{(x,w)_y,(w,z)_y\}-\delta,
\end{equation}
and that $X$ is $\delta$-hyperbolic with respect to the standard
equivalent formulations of hyperbolicity used below.

In a tree, if $x$ belongs to the axis of $v$, then
\[
 d(x,v^nx)=n\tau(v).
\]
The corresponding statement in a hyperbolic space holds up to a
uniformly bounded additive error. More precisely, standard displacement
estimates for hyperbolic isometries
\cite[Chapitre~10, \S6]{CDP90} imply that there exists an absolute
constant $c>0$ such that, for every $x\in A(v)$ and every $n\geq1$,
\begin{equation}
 \bigl|d(x,v^nx)-n\tau(v)\bigr|\leq c\delta.
\end{equation}

We shall also use the following standard consequence of hyperbolicity.
There is a constant $c_0=c_0(\delta)$ such that, whenever
$(x,z)_y>R$, the initial portions of $[y,x]$ and $[y,z]$, up to
distance $R-c_0$ from $y$, fellow travel within distance $c_0$.

We fix a constant $\eta$, depending only on $X$, such that
\[
 \eta\geq\max\{c\delta,10\delta+Q,c_0\},
\]
where $Q$ is the quasiconvexity constant for the sets $A(v)$ fixed
above. In particular, for $x\in A(v)$ and $n\geq1$,
\begin{equation}\label{AxisDisplacementEta}
 \bigl|d(x,v^nx)-n\tau(v)\bigr|\leq\eta.
\end{equation}

We shall also need a coarse analogue of the following elementary
phenomenon for actions on trees. If $v$ acts hyperbolically on a tree,
then its axis and a distinct translate of its axis cannot have
arbitrarily large overlap relative to the translation length of $v$.
In our setting, actual intersection is replaced by coarse overlap: we
consider fixed neighbourhoods of $A(v)$ and of its translate $rA(v)$.
The following uniform periodicity estimate asserts that, unless $r$
commutes with $v$, the diameter of this coarse overlap is bounded by a
uniform multiple of $\tau(v)$.

\begin{fact}\label{Periodicity}
There is a constant $K$, depending only on $X$, such that, if
$v\neq1$, $\tau(v)\geq1$, and $[r,v]\neq1$, then
\begin{equation}\label{PeriodicOverlap}
 \operatorname{diam}
 \bigl(A(v)^{+10\eta}\cap rA(v)^{+10\eta}\bigr)
 \leq K\tau(v).
\end{equation}
We may assume that $K\geq100\eta$.
\end{fact}
\begin{proof}
Apply the cylinder-overlap estimate of
\cite[Lemma~2.30]{DS20}, increasing the hyperbolicity constant by a
fixed factor to accommodate the neighbourhoods in
\eqref{PeriodicOverlap}. For conjugate maximal cyclic subgroups, that
estimate bounds the overlap by a constant multiple of $[C]+\delta$,
where $[C]$ is the least displacement of an element of
$C=C_\Gamma(v)$ whose displacement exceeds the fixed threshold used to define the
cylinder. A suitable power of $v$, together with
\eqref{AxisDisplacementEta}, gives
$[C]\leq\tau(v)+c_1\delta$, with $c_1$ fixed. Since $\tau(v)\geq1$,
the additive terms are absorbed in $K\tau(v)$.
\end{proof}

The particular choice of $10\eta$ is not important. We use this
neighbourhood because it is large enough to accommodate the coarse
geometric errors that arise in the arguments that follow. Replacing
$10\eta$ by another fixed multiple of $\eta$ would only affect the
constant $K$.

We now formulate the small cancellation condition that we shall use.
It is the coarse analogue of the corresponding condition for free
groups. Recall that, in the Cayley tree of a free group, an element is
$N$-small cancellation if it has sufficiently large translation
length and a fundamental segment of its axis has only short
intersections with its nontrivial translates. In the present setting,
the quasi-axis $A(v)$ replaces the axis, while actual intersections
are replaced by coarse intersections of fixed neighbourhoods. Motivated by this analogy, we use the following version of the small cancellation condition for hyperbolic elements (cf. \cite[Definition 4.6]{AF22}).

\begin{definition}\label{def:small-cancellation}
Let $N$ be a positive integer and let $v\in\Gamma$ be nontrivial.
We say that $v$ is $N$-small cancellation (with respect to $X$) if:
\begin{itemize}
    \item $\tau(v)>N$;
    \item for every $f\notin\langle v\rangle$,
    \[
    \Delta(v,fvf^{-1})<\frac{1}{N}\tau(v).
    \]
\end{itemize}
\end{definition}

Notice that the second condition implies
$C_\Gamma(v)=\langle v\rangle$. Indeed, if
$f\in C_\Gamma(v)$, then $fvf^{-1}=v$, and hence
$\Delta(v,fvf^{-1})=\infty$. Thus $f\in\langle v\rangle$.
In particular, an $N$-small cancellation element is not a proper
power.

\subsection{Model companions and existential formulas}\label{sub:model-theory}

Throughout, a first-order theory is assumed to be consistent.

\begin{definition}[Existentially closed substructure]
Let $\mathcal A\subseteq\mathcal B$ be first-order structures in the
same language. We say that $\mathcal A$ is existentially closed in
$\mathcal B$ if for every
existential formula $\phi(\bar{x})$ and every tuple
$\bar a\in A$, if $\mathcal B\models\phi(\bar a)$, then $\mathcal A\models\phi(\bar a)$.
\end{definition}

\begin{definition}[Model complete theory]
A first-order theory $\mathcal{T}$ is model complete if for any two models $\mathcal{A}$, $\mathcal{B}$ of $\mathcal{T}$ with $\mathcal{A}\subseteq\mathcal{B}$ we have that $\mathcal{A}$ is existentially closed in $\mathcal{B}$.
\end{definition}

\begin{fact}(\cite[Theorem 3.5.1]{MR1347464})\label{ModelComplete}
The first-order theory $\mathcal{T}$ is model complete if and only if every first-order formula is equivalent modulo $\mathcal{T}$ to an existential formula.
\end{fact}

Let $\mathcal{T}$ be a first-order theory. We denote by $\mathcal{T}^{\forall}$ the set of all universal first-order sentences which are consequences of $\mathcal{T}$.

\begin{definition}[Model companion]
Let $\mathcal{T}$ be a first-order theory. Then $\mathcal{T}$ admits a model companion if there exists a model complete first-order theory $\mathcal{U}$ such that $\mathcal{T}^{\forall}=\mathcal{U}^{\forall}$. The theory $\mathcal{U}$ is called the model companion of $\mathcal{T}$.
\end{definition}

We now fix a non-elementary torsion-free hyperbolic group $\Gamma$ and
specialize to its first-order theory. By \cite{Sela07} (see also \cite{GLS24}),
\[
\Gamma\equiv\Gamma*\langle z\rangle.
\]
It follows immediately that $\operatorname{Th}(\Gamma)$ is not model
complete. Indeed, consider the inclusion
\[
\Gamma*\langle z^2\rangle\leq\Gamma*\langle z\rangle.
\]
The group $\Gamma*\langle z^2\rangle$ is isomorphic to
$\Gamma*\langle z\rangle$, and hence both groups are models of
$\operatorname{Th}(\Gamma)$. However, the inclusion is not elementary:
the element $z^2$ has a square root in
$\Gamma*\langle z\rangle$, whereas it has no square root in
$\Gamma*\langle z^2\rangle$.

We shall use the following immediate consequence of the completeness of
$\operatorname{Th}(\Gamma)$. If $\mathcal U$ is a theory such that
\[
\mathcal U^\forall=\operatorname{Th}(\Gamma)^\forall
\]
and an existential sentence holds in some model of $\mathcal U$, then
it holds in $\Gamma$, and hence in every model of
$\operatorname{Th}(\Gamma)$. Indeed, otherwise its negation, being
universal, would belong to
$\operatorname{Th}(\Gamma)^\forall=\mathcal U^\forall$, a
contradiction.

Finally, we record a simple but important observation concerning
existential formulas in $\Gamma*\langle z\rangle$. Namely, an
existential formula cannot distinguish $z$ from any of its nonzero
powers. More precisely, if $\phi(x)$ is an existential formula and
\[
\Gamma*\langle z\rangle\models\phi(z),
\]
then
\[
\Gamma*\langle z\rangle\models\phi(z^k)
\]
for every nonzero integer $k$. Indeed, the homomorphism
\[
\Gamma*\langle z\rangle\longrightarrow\Gamma*\langle z\rangle
\]
which restricts to the identity on $\Gamma$ and sends $z$ to $z^k$
is an embedding. Since existential formulas are preserved under
embeddings, the assertion follows.

\section{Uniform Small Cancellation}

The construction underlying our main result is a coarse-geometric
version of the long-power construction for actions on trees
\cite[Proposition~4.8]{MR4206600}.  Its basic mechanism is to concatenate
long powers with distinct exponents, so that sufficiently large
overlaps between two translates of the resulting element force the
corresponding long pieces to align.  In a tree this alignment is
detected by actual intersections of axes.  In a hyperbolic Cayley
graph, both the axes and their intersections are only coarsely
defined, and the main point is to show that the same rigidity survives
with constants uniform in the elements involved.  Fact~\ref{Periodicity}
provides precisely the rigidity needed to recover this alignment from
a long coarse overlap.

The following lemma contains the geometric core of the argument.
It shows that, once a noncommuting pair has been put in a suitable
normal form, the long-power construction produces an element whose
translation length is cubic in the parameter $m$, whereas the overlap
with any translate not arising from one of its own powers is only
quadratic in $m$.

\begin{proposition}\label{HyperbolicSmallCancellation}
There is a positive integer $m_0$, depending only on $X$, with the
following property. Suppose $g,h\in\Gamma$ satisfy
\begin{itemize}
    \item $L=\tau(h)\geq 1$,
    \item $[g,h]\neq 1$, and
    \item there exists $o\in A(h)$ such that $d(o,go)\leq 4L$.
\end{itemize}
For every integer $m\geq m_0$, put $e_i=m^2+i\quad (1\leq i\leq m)$, and set $u=h^{e_1}g\,h^{e_2}g\cdots h^{e_m}g$.

Then
\begin{enumerate}
    \item $\tau(u)\geq \frac12 m^3L$;
    \item if $f\notin\langle u\rangle$, then $\Delta(u,fuf^{-1})\leq 100m^2L$.
\end{enumerate}
\end{proposition}

\begin{proof}
The idea of the proof is as follows. Fix $o\in A(h)$ as in the hypothesis. The word
\[
u=h^{e_1}g\,h^{e_2}g\cdots h^{e_m}g
\]
determines a broken path from $o$ to $uo$. This path consists of long
segments labelled by the powers $h^{e_i}$, separated by short segments
labelled by $g$. We first show that the backtracking between successive
long segments is uniformly bounded. After trimming bounded portions
from their ends, we obtain corresponding successive intervals on a
geodesic joining $u^-$ to $u^+$. Since there are $m$ such intervals,
each of length comparable to $m^2L$, this gives the required lower
bound on $\tau(u)$.

For the second assertion, we show that a sufficiently long overlap
between this geodesic and one of its translates forces three
consecutive translates of the quasi-axis of $h$ to have the same
associated cyclic subgroups. Malnormality, together with the choice of
the exponents, then determines the translating element and shows that
it belongs to $\langle u\rangle$.

We now introduce some auxiliary notation and divide the proof into five
claims. Put
\[
C=C_\Gamma(h).
\]
Since $[g,h]\neq1$, we have $g\notin C$. To keep track of the successive long segments and their translates
under powers of $u$, we repeat the sequence of exponents
$e_1,\ldots,e_m$ in both directions, setting $e_{i+m}=e_i$ for every
$i\in\mathbb Z$. Starting with $P_1=1$, we then define the points
$P_i$ successively by $P_{i+1}=P_i h^{e_i}g$. After $m$ steps we
have multiplied by one full copy of the word $u$. Consequently, this
construction satisfies
\[
P_{i+m}=uP_i, \textrm{ for every } i\in\mathbb Z.
\]

Recall that
\[
d(o,go)\leq4L,
\]
and set
\[
x_i=P_io,\qquad
A_i=P_iA(h),\qquad
C_i=P_iCP_i^{-1},\qquad 
d_i=d(x_i,x_{i+1}).
\]

\medskip
\noindent
\textit{Claim 1.}
There exists a constant $b$, depending only on $X$, such that
\begin{equation}\label{LocalEstimates}
\left|d_i-e_iL\right|\leq4L+\eta
\qquad\text{and}\qquad
(x_{i-1},x_{i+1})_{x_i}\leq bL.
\end{equation}

\smallskip
\noindent\emph{Proof of Claim 1.}
By the axis-displacement estimate \eqref{AxisDisplacementEta}, there is
a constant $\eta$, depending only on $X$, such that
\[
\bigl|d(o,h^{e_i}o)-e_iL\bigr|\leq\eta,
\]
On the other hand, the hypothesis
$d(o,go)\leq4L$, together with the fact that $h^{e_i}$ acts by an
isometry, gives
\[
d(h^{e_i}o,h^{e_i}go)=d(o,go)\leq4L.
\]
Since
\[
d_i=d(o,h^{e_i}go),
\]
the triangle inequality and its reverse form imply
\[
\bigl|d_i-d(o,h^{e_i}o)\bigr|\leq4L.
\]
Combining the two estimates yields
\[
\bigl|d_i-e_iL\bigr|\leq4L+\eta.
\]

We now prove the second inequality. Since the Gromov product is
invariant under isometries, we translate the configuration by
$P_i^{-1}$, so that $x_i=P_io$ is moved to $o$. Set
\[
z_-=P_i^{-1}x_{i-1}=g^{-1}h^{-e_{i-1}}o,
\qquad
z_+=P_i^{-1}x_{i+1}=h^{e_i}go.
\]
Then
\[
(x_{i-1},x_{i+1})_{x_i}=(z_-,z_+)_o.
\]
The advantage of this normalization is that the two geodesics issuing
from $o$ can be compared directly with $g^{-1}A(h)$ and $A(h)$,
respectively. We shall show that, apart from uniformly bounded
portions at their ends, $[o,z_-]$ lies close to $g^{-1}A(h)$ and
$[o,z_+]$ lies close to $A(h)$. A large value of
$(z_-,z_+)_o$ would then force a long coarse intersection between
these two quasi-axes, contradicting \eqref{PeriodicOverlap}.

We make these statements precise. Consider first the geodesic triangle with vertices
\[
o,\qquad h^{e_i}o,\qquad z_+=h^{e_i}go.
\]
Since
\[
d(h^{e_i}o,z_+)=d(o,go)\leq 4L,
\]
let $p\in[o,z_+]$ satisfy
\[
d(p,z_+)>4L+\eta.
\]
By $\delta$-thinness, $p$ lies within distance $\delta$ of either
$[o,h^{e_i}o]$ or $[h^{e_i}o,z_+]$. The latter is impossible.
Indeed, if $q\in[h^{e_i}o,z_+]$ satisfies $d(p,q)\leq\delta$, then
\[
d(p,z_+)
\leq d(p,q)+d(q,z_+)
\leq \delta+4L
\leq \eta+4L,
\]
a contradiction. Hence $p$ lies within distance $\delta$ of
$[o,h^{e_i}o]$.

Now $o,h^{e_i}o\in A(h)$, and $A(h)$ is $Q$-quasiconvex. Therefore
\[
[o,h^{e_i}o]\subseteq A(h)^{+Q}.
\]
It follows that
\[
d(p,A(h))\leq \delta+Q\leq\eta,
\]
by the choice of $\eta$. Thus
\begin{equation}\label{PositiveAxisEstimate}
[o,z_+]\setminus B(z_+,4L+\eta)
\subseteq A(h)^{+\eta}.
\end{equation}

Similarly, consider the geodesic triangle with vertices
\[
o,\qquad g^{-1}o,\qquad z_-=g^{-1}h^{-e_{i-1}}o.
\]
Since
\[
d(o,g^{-1}o)=d(o,go)\leq4L,
\]
let $p\in[o,z_-]$ satisfy
\[
d(o,p)>4L+\eta.
\]
By $\delta$-thinness, $p$ lies within distance $\delta$ of either
$[o,g^{-1}o]$ or $[g^{-1}o,z_-]$. The former is impossible.
Indeed, if $q\in[o,g^{-1}o]$ satisfies $d(p,q)\leq\delta$, then
\[
d(o,p)
\leq d(o,q)+d(q,p)
\leq4L+\delta
\leq4L+\eta,
\]
a contradiction. Hence $p$ lies within distance $\delta$ of
$[g^{-1}o,z_-]$.

Both endpoints of $[g^{-1}o,z_-]$ belong to $g^{-1}A(h)$:
indeed, $o,h^{-e_{i-1}}o\in A(h)$. Since $A(h)$ is
$Q$-quasiconvex, so is its translate $g^{-1}A(h)$, and therefore
\[
[g^{-1}o,z_-]\subseteq \bigl(g^{-1}A(h)\bigr)^{+Q}.
\]
Consequently,
\[
d\bigl(p,g^{-1}A(h)\bigr)\leq\delta+Q\leq\eta.
\]
Thus
\begin{equation}\label{NegativeAxisEstimate}
[o,z_-]\setminus B(o,4L+\eta)
\subseteq \bigl(g^{-1}A(h)\bigr)^{+\eta}.
\end{equation}

We now bound the Gromov product. Let $b>0$ be a constant, to be fixed
below, and put
\[
\beta=bL.
\]
For any fixed $b$, since $e_i\geq m^2$, the first inequality in
\eqref{LocalEstimates} gives
\[
d(o,z_+)\geq m^2L-4L-\eta.
\]
Thus we may choose $m_0$, depending only on $X$ and $b$, so that, for
$m\geq m_0$,
\begin{equation}\label{LongPositiveSegment}
d(o,z_+)>\beta+4L+\eta.
\end{equation}

Suppose now that $(z_-,z_+)_o>\beta$. By our choice of $\eta$, the
initial portions of $[o,z_-]$ and $[o,z_+]$, up to distance
$\beta-\eta$ from $o$, fellow travel within distance $\eta$.
Discarding from the former its initial portion of length $4L+\eta$,
\eqref{NegativeAxisEstimate} shows that the remaining segment lies in
$\bigl(g^{-1}A(h)\bigr)^{+\eta}$. The corresponding portion of
$[o,z_+]$ therefore lies in
$\bigl(g^{-1}A(h)\bigr)^{+2\eta}$. Moreover, our choice of $m_0$
ensures, by \eqref{LongPositiveSegment}, that this portion is disjoint
from $B(z_+,4L+\eta)$. Hence \eqref{PositiveAxisEstimate} shows that
it also lies in $A(h)^{+\eta}$.

Consequently,
\[
A(h)^{+10\eta}\cap
\bigl(g^{-1}A(h)\bigr)^{+10\eta}
\]
has diameter at least
\[
\beta-4L-2\eta.
\]
Since $\beta=bL$ and $L\geq1$, we may now fix $b$, depending only on
$X$, sufficiently large that this quantity is greater than $KL$.
This contradicts \eqref{PeriodicOverlap}, since $g^{-1}\notin C$.

Hence
\[
(z_-,z_+)_o\leq bL,
\]
and therefore
\[
(x_{i-1},x_{i+1})_{x_i}\leq bL.
\]
\qed

Thus Claim~1 gives uniform control of the geometry at each junction of
the broken path determined by the word $u$: the length of each long
piece is $e_iL$ up to an error of order $L$, while the backtracking
between two successive pieces is also bounded by a constant multiple
of $L$. Since $e_i\asymp m^2$, the long pieces become arbitrarily
large compared with these errors as $m$ grows. We shall now use this
local control to obtain a global description of the broken path: after
trimming bounded portions from its long pieces, they occur in order
near a geodesic joining $u^-$ to $u^+$. 

Before moving to claim 2 recall that

\[
P_{i+1}=P_i h^{e_i}g, \qquad
x_i=P_io,\qquad
A_i=P_iA(h),\qquad
d_i=d(x_i,x_{i+1}).
\]

\ \\ \\
\noindent\emph{Claim 2.}
There exist a constant $B$, depending only on $X$, a geodesic line
$\ell$ joining $u^-$ to $u^+$, and pairwise disjoint intervals
$J_i\subseteq\ell$, indexed by $i\in\mathbb Z$, with the following
properties. As $\ell$ is traversed from $u^-$ to $u^+$, the intervals
$J_i$ occur in increasing order of $i$. For every $i\in\mathbb Z$,
\[
 J_i\subseteq A_i^{+\eta}
\]
and
\[
 (e_i-B)L\leq |J_i|\leq(e_i+B)L.
\]
Moreover, the segment of $\ell$ between $J_i$ and $J_{i+1}$ has
length at most $BL$. The orientation of $J_i$ from $u^-$ to $u^+$
agrees with the orientation of $A_i=P_iA(h)$ from $P_i h^-$ to
$P_i h^+$.

Finally,
\begin{equation}\label{LengthBeforeChoice}
 \tau(u)\geq
 \left(\sum_{i=1}^m e_i-Bm\right)L.
\end{equation}

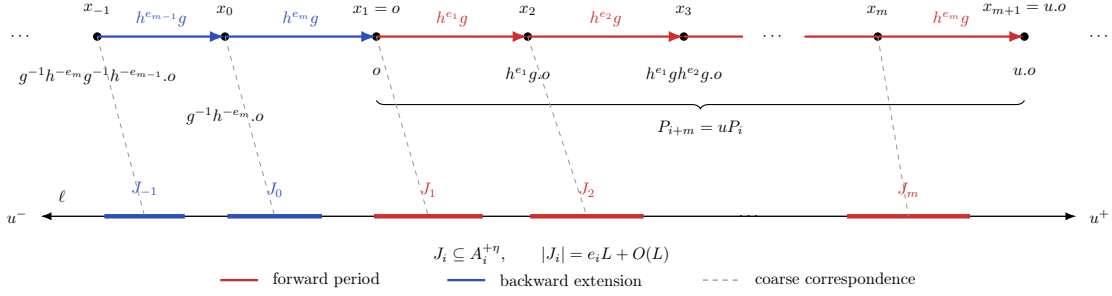
\begin{figure}[ht]
\centering

\resizebox{0.95\textwidth}{!}{%
\begin{tikzpicture}[
    x=1cm,
    y=1cm,
    >=Latex,
    every node/.style={font=\small},
    point/.style={circle,fill=black,inner sep=1.8pt},
    fwd/.style={->,line width=1.35pt,forwardred},
    bwd/.style={->,line width=1.35pt,backwardblue},
    guide/.style={dashed,softgray,line width=.5pt}
]


\coordinate (xm1) at (-6.4,3.7);
\coordinate (x0)  at (-3.7,3.7);
\coordinate (x1)  at (-0.5,3.7);
\coordinate (x2)  at ( 2.7,3.7);
\coordinate (x3)  at ( 6.0,3.7);
\coordinate (xm)  at (10.1,3.7);
\coordinate (xmp) at (13.2,3.7);

\node at (-8.0,3.7) {$\cdots$};
\node at (7.9,3.7) {$\cdots$};
\node at (14.8,3.7) {$\cdots$};

\foreach \p in {xm1,x0,x1,x2,x3,xm,xmp}
  \node[point] at (\p) {};


\node[above=9pt] at (xm1) {$x_{-1}$};
\node[above=9pt] at (x0)  {$x_0$};
\node[above=9pt] at (x1)  {$x_1=o$};
\node[above=9pt] at (x2)  {$x_2$};
\node[above=9pt] at (x3)  {$x_3$};
\node[above=9pt] at (xm)  {$x_m$};
\node[above=9pt] at (xmp) {$x_{m+1}=u.o$};


\draw[bwd]
  (xm1) -- node[midway,above=2pt,text=backwardblue]
  {$h^{e_{m-1}}g$} (x0);

\draw[bwd]
  (x0) -- node[midway,above=2pt,text=backwardblue]
  {$h^{e_m}g$} (x1);


\draw[fwd]
  (x1) -- node[midway,above=2pt,text=forwardred]
  {$h^{e_1}g$} (x2);

\draw[fwd]
  (x2) -- node[midway,above=2pt,text=forwardred]
  {$h^{e_2}g$} (x3);

\draw[line width=1.35pt,forwardred]
  (x3) -- (7.25,3.7);

\draw[line width=1.35pt,forwardred]
  (8.55,3.7) -- (xm);

\draw[fwd]
  (xm) -- node[midway,above=2pt,text=forwardred]
  {$h^{e_m}g$} (xmp);

%

\node[below=15pt,anchor=north,align=center] at (xm1)
 {$g^{-1}h^{-e_m}g^{-1}h^{-e_{m-1}}.o$};

\node[below=39pt,anchor=north] at (x0)
 {$g^{-1}h^{-e_m}.o$};

\node[below=15pt] at (x1) {$o$};

\node[below=15pt] at (x2)
 {$h^{e_1}g.o$};

\node[below=15pt] at (x3)
 {$h^{e_1}g h^{e_2}g.o$};

\node[below=15pt] at (xmp)
 {$u.o$};


\draw[
 decorate,
 decoration={brace,amplitude=5pt,mirror},
 line width=.65pt
]
(-0.5,2.35) -- (13.2,2.35)
node[midway,below=8pt]
{$P_{i+m}=uP_i$};


\coordinate (ellL) at (-7.6,-0.1);
\coordinate (ellR) at (14.3,-0.1);

\draw[<->,line width=.8pt] (ellL) -- (ellR);

\node[left=4pt]  at (ellL) {$u^-$};
\node[right=4pt] at (ellR) {$u^+$};

\node[above=5pt] at (-7.15,-0.1) {$\ell$};


\draw[line width=2.7pt,backwardblue]
  (-6.25,-0.1) -- (-4.55,-0.1);
\node[above=8pt,text=backwardblue] at (-5.4,-0.1)
  {$J_{-1}$};

\draw[line width=2.7pt,backwardblue]
  (-3.65,-0.1) -- (-1.65,-0.1);
\node[above=8pt,text=backwardblue] at (-2.65,-0.1)
  {$J_0$};

\draw[line width=2.7pt,forwardred]
  (-0.55,-0.1) -- (1.75,-0.1);
\node[above=8pt,text=forwardred] at (0.6,-0.1)
  {$J_1$};

\draw[line width=2.7pt,forwardred]
  (2.75,-0.1) -- (5.15,-0.1);
\node[above=8pt,text=forwardred] at (3.95,-0.1)
  {$J_2$};

\node at (7.4,-0.1) {$\cdots$};

\draw[line width=2.7pt,forwardred]
  (9.45,-0.1) -- (12.05,-0.1);
\node[above=8pt,text=forwardred] at (10.75,-0.1)
  {$J_m$};


\draw[guide] (xm1) -- (-5.4,-0.1);
\draw[guide] (x0)  -- (-2.65,-0.1);
\draw[guide] (x1)  -- (0.6,-0.1);
\draw[guide] (x2)  -- (3.95,-0.1);
\draw[guide] (xm)  -- (10.75,-0.1);


\node[below=13pt] at (3.2,-0.1)
 {$J_i\subseteq A_i^{+\eta},
   \qquad
   |J_i|=e_iL+O(L)$};


\begin{scope}[yshift=-1.45cm]

\draw[line width=1.4pt,forwardred]
 (-3.8,0) -- (-3.0,0);
\node[right=5pt] at (-3.0,0)
 {forward period};

\draw[line width=1.4pt,backwardblue]
 (1.0,0) -- (1.8,0);
\node[right=5pt] at (1.8,0)
 {backward extension};

\draw[dashed,line width=.8pt,softgray]
 (6.4,0) -- (7.2,0);
\node[right=5pt] at (7.2,0)
 {coarse correspondence};

\end{scope}

\end{tikzpicture}
}
\caption{The bi-infinite broken path associated with
$u=h^{e_1}g\cdots h^{e_m}g$ and the ordered intervals $J_i$
on a geodesic line $\ell$ joining $u^-$ to $u^+$.
After uniformly bounded portions are removed from the ends of
$[x_i,x_{i+1}]$, the retained portions fellow travel the
corresponding intervals $J_i\subseteq A_i^{+\eta}$.}
\label{fig:broken-path}

\end{figure}

\smallskip
\noindent\emph{Proof of Claim 2.}
We first promote the local control given by Claim~1 to uniform control
of arbitrary subchains of the broken path.  More precisely, we shall
show that, viewed from any intermediate vertex $x_i$, the two ends of
a subchain extending on either side of $x_i$ have uniformly bounded
Gromov product.  This will imply that the broken path makes definite
progress: the distance between two of its vertices is bounded below
by the sum of the lengths of the intervening pieces, up to a uniformly
bounded loss at each junction.

Increasing $m_0$ if necessary, we may assume that
\[
 d_i>10\beta+100\delta, \textrm{ for every } i.
\]
We first claim that
\begin{equation}\label{OneSidedProduct}
 (x_r,x_{i+1})_{x_i}\leq\beta+\delta
 \qquad(r<i).
\end{equation}
Thus the bound on the backtracking at $x_i$ given by Claim~1 remains
uniform even when the preceding vertex $x_{i-1}$ is replaced by an
arbitrarily distant vertex $x_r$ of the broken path.

We prove \eqref{OneSidedProduct} by induction on $i-r$.  The case
$r=i-1$ is precisely Claim~1.  For the inductive step, the induction
hypothesis gives
\[
 (x_r,x_i)_{x_{i-1}}\leq\beta+\delta.
\]
Hence
\[
 (x_r,x_{i-1})_{x_i}
 =d_{i-1}-(x_r,x_i)_{x_{i-1}}
 \geq d_{i-1}-\beta-\delta
 >\beta+2\delta.
\]
Together with
\[
 (x_{i-1},x_{i+1})_{x_i}\leq\beta,
\]
the four-point inequality \eqref{FourPoint} now gives
\eqref{OneSidedProduct}.

Applying the same argument in the reverse direction gives
\[
 (x_{i-1},x_t)_{x_i}\leq\beta+\delta
 \qquad(t>i).
\]
Now let $r<i<t$.  As in the induction above,
\[
 (x_r,x_{i-1})_{x_i}>\beta+2\delta.
\]
On the other hand, the four-point inequality gives
\[
 \min\bigl\{(x_r,x_{i-1})_{x_i},
            (x_r,x_t)_{x_i}\bigr\}
 \leq (x_{i-1},x_t)_{x_i}+\delta
 \leq\beta+2\delta.
\]
The first term in the minimum is strictly larger than
$\beta+2\delta$, and hence
\begin{equation}\label{TwoSidedProduct}
 (x_r,x_t)_{x_i}\leq\beta+2\delta
 \qquad(r<i<t).
\end{equation}
In other words, the amount of backtracking at an intermediate vertex
remains uniformly bounded independently of how far the two endpoints
of the subchain are chosen.

This also gives a lower bound on the total progress along the broken
path.  Indeed, from \eqref{OneSidedProduct},
\[
 d(x_r,x_{i+1})
 \geq d(x_r,x_i)+d_i-2(\beta+\delta).
\]
Iterating this inequality for $i=r+1,\ldots,t-1$ yields
\begin{equation}\label{ChainLength}
 d(x_r,x_t)\geq
 \sum_{i=r}^{t-1}d_i
 -2(\beta+\delta)(t-r-1).
\end{equation}
Thus the only loss, compared with the total length of the broken path,
is a uniformly bounded amount at each junction.

We now remove the portions of the broken path near the vertices, where
the bounded backtracking detected above may occur.  Put
\[
 R=\beta+4L+10\eta,
\]
and let $I_i$ be the subsegment of $[x_i,x_{i+1}]$ obtained by
removing a segment of length $R$ from each end.  Increasing $m_0$ if
necessary, we may assume that $d_i>2R$ for every $i$, so that each
$I_i$ is nonempty.

We first compare $I_i$ with geodesics joining distant vertices of the
broken path.  Fix
\[
 r<i<i+1<t.
\]

Consider first the geodesic triangle with vertices
$x_r,x_i,x_{i+1}$.  By \eqref{OneSidedProduct},
\[
 (x_r,x_{i+1})_{x_i}\leq\beta+\delta.
\]
We claim that
\[
 I_i\subseteq[x_r,x_{i+1}]^{+\delta}.
\]
Indeed, let $p\in I_i$ and put $s=d(x_i,p)$.  By
$\delta$-thinness, $p$ lies within distance $\delta$ of either
$[x_i,x_r]$ or $[x_r,x_{i+1}]$.  If
$d(p,[x_i,x_r])\leq\delta$, then the definition of the Gromov
product gives
\[
 (x_r,x_{i+1})_{x_i}\geq s-\delta.
\]
Consequently $s\leq\beta+2\delta$.  On the other hand,
$p\in I_i$ implies $s\geq R>\beta+2\delta$, a contradiction.
Hence $p$ lies within distance $\delta$ of
$[x_r,x_{i+1}]$, proving the claim.

We now consider the geodesic triangle with vertices
$x_r,x_{i+1},x_t$.  We claim that
\begin{equation}\label{IiCloseToLongGeodesic}
 I_i\subseteq[x_r,x_t]^{+2\delta}.
\end{equation}
Indeed, let $p\in I_i$.  By the preceding argument, there exists
$q\in[x_r,x_{i+1}]$ such that
\[
 d(p,q)\leq\delta.
\]
By $\delta$-thinness of the triangle
$(x_r,x_{i+1},x_t)$, the point $q$ lies within distance $\delta$
of either $[x_{i+1},x_t]$ or $[x_r,x_t]$.

Suppose first that $d(q,[x_{i+1},x_t])\leq\delta$.  Applying the
same Gromov-product estimate as above, now with basepoint
$x_{i+1}$, gives
\[
 (x_r,x_t)_{x_{i+1}}
 \geq d(x_{i+1},q)-\delta.
\]
On the other hand, \eqref{TwoSidedProduct}, applied with $i+1$ in
place of $i$, gives
\[
 (x_r,x_t)_{x_{i+1}}\leq\beta+2\delta.
\]
Hence
\[
 d(x_{i+1},q)\leq\beta+3\delta.
\]
Since $p\in I_i$, however,
\[
 d(x_{i+1},p)\geq R,
\]
and therefore
\[
 d(x_{i+1},q)\geq R-\delta>\beta+3\delta,
\]
a contradiction.  Thus $q$ lies within distance $\delta$ of
$[x_r,x_t]$.  Since $d(p,q)\leq\delta$, we conclude that
\[
 d(p,[x_r,x_t])\leq2\delta,
\]
which proves \eqref{IiCloseToLongGeodesic}.
We next compare consecutive retained segments.  Write
\[
 I_i=[a_i,b_i],
\]
where $a_i$ is the endpoint closer to $x_i$ and $b_i$ the endpoint
closer to $x_{i+1}$.  By \eqref{IiCloseToLongGeodesic}, we may choose
$a_i',b_i'\in[x_r,x_t]$ such that
\[
 d(a_i,a_i'),\,d(b_i,b_i')\leq2\delta,
\]
and similarly for $I_{i+1}$.

We first estimate the separation between two consecutive retained
segments before passing to $[x_r,x_t]$.  Put $v=x_{i+1}$.  By
construction,
\[
 d(b_i,v)=d(v,a_{i+1})=R.
\]
Since $b_i\in[x_i,v]$ and $a_{i+1}\in[v,x_{i+2}]$, we have
\[
\begin{aligned}
 d(x_i,x_{i+2})
 &\leq d(x_i,b_i)+d(b_i,a_{i+1})
       +d(a_{i+1},x_{i+2})  \\
 &=d_i+d_{i+1}-2R+d(b_i,a_{i+1}).
\end{aligned}
\]
On the other hand,
\[
 d(x_i,x_{i+2})
 =d_i+d_{i+1}
   -2(x_i,x_{i+2})_{x_{i+1}},
\]
and Claim~1 gives
\[
 (x_i,x_{i+2})_{x_{i+1}}\leq\beta.
\]
Consequently,
\[
 d(b_i,a_{i+1})\geq2R-2\beta.
\]
The reverse inequality
\[
 d(b_i,a_{i+1})\leq2R
\]
follows immediately by joining $b_i$ to $a_{i+1}$ through
$x_{i+1}$.  Hence
\begin{equation}\label{GapBeforeStraightening}
 2R-2\beta
 \leq d(b_i,a_{i+1})
 \leq2R.
\end{equation}

The construction above also shows that the points on $[x_r,x_t]$
which fellow travel the successive retained segments occur in the
same order: the point $b_i'$ occurs before $a_{i+1}'$ as
$[x_r,x_t]$ is traversed from $x_r$ to $x_t$. More explicitly, comparing distances from $x_r$ and using the
Gromov-product estimates above gives
\[
 d(x_r,a_{i+1}')-d(x_r,b_i')
 \geq 2R-2\beta-12\delta>0,
\]
which proves the asserted order.  Thus the corresponding
intervals on $[x_r,x_t]$ are disjoint and occur in their natural
order.  Since
\[
 d(b_i,b_i'),\,d(a_{i+1},a_{i+1}')\leq2\delta,
\]
\eqref{GapBeforeStraightening} gives
\begin{equation}\label{GapAfterStraightening}
 2R-2\beta-4\delta
 \leq d(b_i',a_{i+1}')
 \leq2R+4\delta.
\end{equation}
In particular, since
\[
 R=\beta+4L+10\eta,
\]
the lower bound in \eqref{GapAfterStraightening} is positive.
Thus consecutive retained segments remain separated after being
straightened onto $[x_r,x_t]$, while the gaps between them are
uniformly bounded by $2R+4\delta$.


We now pass from these finite geodesics to a single bi-infinite
geodesic which records the geometry of the entire broken path.  We
shall first show that the two ends of the sequence $(x_i)_{i\in
\mathbb Z}$ determine distinct points of $\partial X$, and then obtain
the intervals $J_i$ by passing the retained segments $I_i$ to a
limiting geodesic between these boundary points.

Fix the basepoint $o=x_1$.  We first consider the positive direction.
By Claim~1,
\[
 d_i\geq e_iL-4L-\eta,
\]
and $e_i\geq m^2$.  Hence, for $m$ sufficiently large,
\eqref{ChainLength} implies
\[
 d(o,x_i)\longrightarrow\infty
 \qquad\text{as }i\longrightarrow+\infty.
\]
Moreover, if $1<i<n<k$, then \eqref{TwoSidedProduct} gives
\[
 (o,x_n)_{x_i},\,(o,x_k)_{x_i}\leq\beta+2\delta.
\]
Using the definition of the Gromov product, it follows that
\[
 (x_n,x_k)_o
 \geq d(o,x_i)-2(\beta+2\delta).
\]
Letting first $n,k\to+\infty$ and then $i\to+\infty$, we obtain
\[
 (x_n,x_k)_o\longrightarrow\infty.
\]
Thus $(x_i)_{i\geq1}$ is a Gromov sequence and determines a point
$\xi_+\in\partial X$.  The same argument in the negative direction
gives a point $\xi_-\in\partial X$ such that
\[
 x_i\longrightarrow\xi_-
 \qquad\text{as }i\longrightarrow-\infty.
\]
The estimate \eqref{TwoSidedProduct}, with $r<1<t$, shows that
$\xi_-\neq\xi_+$.

By properness of $X$, after passing to a subsequence, the geodesics
$[x_r,x_t]$, with $r\to-\infty$ and $t\to+\infty$, converge uniformly
on compact sets to a bi-infinite geodesic $\ell$ joining
$\xi_-$ to $\xi_+$.  By the periodic construction of the prefixes,
\[
 P_{i+m}=uP_i,
\]
and hence
\[
 x_{i+m}=u.x_i.
\]
In particular, since $x_1=o$,
\[
 x_{nm+1}=u^n.o.
\]
It follows that
\[
 \xi_+=u^+,\qquad \xi_-=u^-,
\]
and hence $\ell$ is a geodesic joining $u^-$ to $u^+$.

We now pass the retained segments to this limiting geodesic.  For
each fixed $i$, the preceding argument shows that $I_i$ lies within
distance $2\delta$ of every sufficiently large geodesic $[x_r,x_t]$
with $r<i<i+1<t$.  The corresponding intervals on $[x_r,x_t]$ occur
in their natural order, and the gaps between consecutive ones are
uniformly bounded.  Passing to the above subsequence, these intervals
converge to intervals
\[
 J_i\subseteq\ell,\qquad i\in\mathbb Z,
\]
which are pairwise disjoint and occur in increasing order of $i$ as
$\ell$ is traversed from $u^-$ to $u^+$.

Since $|I_i|=d_i-2R$ and the endpoints of the corresponding interval
on $[x_r,x_t]$ may be chosen within $2\delta$ of the endpoints of
$I_i$, passage to the limit gives
\[
 \bigl||J_i|-(d_i-2R)\bigr|\leq4\delta.
\]
Similarly, the estimates above give
\[
 \operatorname{gap}(J_i,J_{i+1})\leq2R+4\delta.
\]

It remains to verify the geometric and quantitative properties of the
intervals $J_i$.  Put
\[
 y_i=P_i h^{e_i}.o.
\]
Then $x_i,y_i\in A_i$, while
\[
 d(y_i,x_{i+1})=d(o,g.o)\leq4L.
\]
Since $A_i$ is uniformly quasiconvex,
\[
 [x_i,y_i]\subseteq A_i^{+Q}.
\]
Moreover, a thin-triangle argument applied to the triangle with
vertices $x_i,y_i,x_{i+1}$ shows that, after removing the terminal
portion of length $R$ from $[x_i,x_{i+1}]$, the remaining segment
lies within distance $\delta$ of $[x_i,y_i]$.  In particular,
\[
 I_i\subseteq A_i^{+(Q+\delta)}.
\]
The constants in the initial choice of $\eta$ were chosen with enough
slack to absorb the additional $2\delta$ arising when $I_i$ is
straightened onto $\ell$.  Hence
\[
 J_i\subseteq A_i^{+\eta}.
\]
Since $e_i>0$, the segment $[x_i,y_i]$ runs in the positive direction
of $A_i$, from $P_i h^-$ towards $P_i h^+$.  For $m$ sufficiently
large the retained segment has length much greater than all the
bounded endpoint errors, so the orientation of $J_i$ from $u^-$ to
$u^+$ agrees with this orientation.

We next estimate lengths.  Since $I_i$ is obtained from
$[x_i,x_{i+1}]$ by removing a segment of length $R$ from each end,
\[
 |I_i|=d_i-2R.
\]
The construction of $J_i$ gives
\[
 \bigl||J_i|-(d_i-2R)\bigr|\leq4\delta.
\]
Together with Claim~1,
\[
 |d_i-e_iL|\leq4L+\eta,
\]
and the definition $R=\beta+4L+10\eta$, this yields
\[
 \bigl||J_i|-e_iL\bigr|\leq B L
\]
for some constant $B=B(X)$.  Enlarging $B$ if necessary, the gap
estimate obtained above also gives
\[
 \operatorname{gap}(J_i,J_{i+1})\leq BL.
\]
This proves all the assertions concerning the intervals $J_i$.

It remains only to obtain the lower bound for $\tau(u)$.  By
periodicity,
\[
 x_{nm+1}=u^n.o.
\]
Applying \eqref{ChainLength} with $r=1$ and $t=nm+1$ gives
\[
 d(o,u^n.o)
 \geq
 \sum_{i=1}^{nm}d_i
 -2(\beta+\delta)(nm-1).
\]
The sequence $(d_i)$ is $m$-periodic, and Claim~1 gives
\[
 d_i\geq e_iL-4L-\eta.
\]
Dividing by $n$ and letting $n\to\infty$, we obtain, after enlarging
$B=B(X)$ once more,
\[
 \tau(u)\geq
 \left(\sum_{i=1}^m e_i-Bm\right)L.
\]
This is \eqref{LengthBeforeChoice} and completes the proof of
Claim~2. \qed

\ \\ \\
We now fix $m_0\geq10$ once and for all, large enough to satisfy all
the lower bounds imposed in Claims~1 and~2, and in addition so that
\begin{equation}\label{Choice}
 m_0^2\geq100(B+K+\eta+1).
\end{equation}
For every $m\geq m_0$, since
\[
 \sum_{i=1}^m e_i
 =m^3+\frac{m(m+1)}2,
\]
\eqref{LengthBeforeChoice} gives
\[
 \tau(u)\geq
 \left(m^3+\frac{m(m+1)}2-Bm\right)L
 \geq\frac12m^3L.
\]
Thus the first assertion of the lemma is proved.

For the remainder of the proof, put
\[
 s=m^2L.
\]
This is the natural length scale of the intervals $J_i$.  By Claim~2,
\[
 (e_i-B)L\leq |J_i|\leq(e_i+B)L,
\]
and the gap between $J_i$ and $J_{i+1}$ has length at most $BL$.
Since $e_i=m^2+i$ and $1\leq i\leq m$, our choice
\eqref{Choice} gives
\begin{equation}\label{Comparison}
 \frac{s}{2}\leq |J_i|\leq2s,\qquad
 BL\leq\frac{s}{100},\qquad
 KL\leq\frac{s}{100},\qquad
 \eta\leq\frac{s}{100}.
\end{equation}
Thus each $J_i$ has length comparable to $s$, whereas the gaps between
consecutive intervals and the other geometric error terms are bounded
by a fixed small fraction of $s$.

\ \\ \\
\noindent\emph{Claim 3.}
The subgroups $C_i$, $i\in\mathbb Z$, are pairwise distinct.

\smallskip
\noindent\emph{Proof of Claim 3.}
We first show that consecutive subgroups are distinct.  Recall that
\[
 C_i=P_iCP_i^{-1},
 \qquad
 P_{i+1}=P_i h^{e_i}g.
\]
Since $h^{e_i}\in C$, the equality $C_i=C_{i+1}$ would imply
\[
 C=gCg^{-1}.
\]
Thus $g$ normalizes $C$.  Since $C$ is malnormal, and hence
self-normalizing, this would give $g\in C$, contrary to
$[g,h]\neq1$.  Hence
\begin{equation}\label{ConsecutiveDistinct}
 C_i\neq C_{i+1}
 \qquad\text{for every }i\in\mathbb Z.
\end{equation}

Suppose now, towards a contradiction, that
\[
 C_i=C_j
\]
for some $j\geq i+2$.  The subgroup $C_i$ stabilizes the pair
$\{P_i h^-,P_i h^+\}$, and similarly for $C_j$.  Thus the equality
$C_i=C_j$ implies that these pairs of endpoints coincide, and hence,
by the definition of the quasi-axes,
\[
 A_i=A_j.
\]

By Claim~2,
\[
 J_i,J_j\subseteq A_i^{+\eta}.
\]
Choose $p\in J_i$ and $q\in J_j$, and choose $p',q'\in A_i$ with
\[
 d(p,p'),\,d(q,q')\leq\eta.
\]
Since $A_i$ is $Q$-quasiconvex,
\[
 [p',q']\subseteq A_i^{+Q}.
\]
The geodesic quadrilateral with vertices $p,p',q',q$ is
$2\delta$-thin, and therefore
\[
 [p,q]\subseteq
 A_i^{+\max\{\eta+2\delta,Q+2\delta\}}
 \subseteq A_i^{+10\eta}.
\]
Choosing $p$ and $q$ at the appropriate ends of $J_i$ and $J_j$,
and using the order of the intervals along $\ell$, we obtain
\[
 J_{i+1}\subseteq[p,q]\subseteq A_i^{+10\eta}.
\]
On the other hand, Claim~2 gives
\[
 J_{i+1}\subseteq A_{i+1}^{+\eta}
 \subseteq A_{i+1}^{+10\eta}.
\]
Hence
\[
 J_{i+1}\subseteq
 A_i^{+10\eta}\cap A_{i+1}^{+10\eta},
\]
and therefore
\[
 \operatorname{diam}
 \bigl(A_i^{+10\eta}\cap A_{i+1}^{+10\eta}\bigr)
 \geq |J_{i+1}|
 \geq\frac{s}{2}>KL.
\]
Conjugating by $P_i^{-1}$, this gives
\[
 \operatorname{diam}
 \bigl(
 A(h)^{+10\eta}
 \cap h^{e_i}gA(h)^{+10\eta}
 \bigr)>K\tau(h).
\]
The Periodicity Fact, Fact \ref{Periodicity}, therefore implies
\[
 [h,h^{e_i}g]=1.
\]
Since $h^{e_i}$ commutes with $h$, this gives $[h,g]=1$, contrary
to our hypothesis.  This contradiction proves that the subgroups
$C_i$ are pairwise distinct.
\qed
\ \\ \\

\noindent\emph{Claim 4.}
If
\begin{equation}\label{LongOverlap}
 \Delta(u,fuf^{-1})>100s,
\end{equation}
then there exist $i,j\in\mathbb Z$ such that
\begin{equation}\label{ThreeMatches}
 fC_{j+r}f^{-1}=C_{i+r}
 \qquad(r=0,1,2).
\end{equation}

\smallskip
\noindent\emph{Proof of Claim 4.}
We first extract from \eqref{LongOverlap} a long segment on which
$\ell$ and $f\ell$ fellow travel.  Recall that $\ell$ joins
$u^-$ to $u^+$, while $f\ell$ joins $fu^-$ to $fu^+$.  By the definition of $\Delta$, there are two points
\[
 z_-,z_+\in
 A(u)^{+100\delta}\cap fA(u)^{+100\delta}
\]
with $d(z_-,z_+)>100s$.  Choosing nearby points on $\ell$ and on
$f\ell$, and applying thinness of geodesic quadrilaterals, we obtain
long subsegments of $\ell$ and $f\ell$ which fellow travel.  More
precisely, after allowing for uniformly bounded errors near their
endpoints, there is an interval $D$, with
\[
 |D|>99s,
\]
and isometric parametrizations
\[
 \alpha:D\longrightarrow\ell,
 \qquad
 \alpha_f:D\longrightarrow f\ell
\]
such that
\begin{equation}\label{FellowTravel}
 d(\alpha(t),\alpha_f(t))\leq\eta
 \qquad(t\in D).
\end{equation}

We use \eqref{FellowTravel} to compare the two decompositions of this
fellow-travelling corridor by the intervals $J_i$ and $fJ_j$.
Suppose that, for some $i,j$, a subinterval of $D$ corresponds on
one side to $J_i$ and on the other to $fJ_j$.  Since
\[
 J_i\subseteq A_i^{+\eta},
 \qquad
 fJ_j\subseteq fA_j^{+\eta},
\]
the corresponding points on $\ell$ lie in
\[
 A_i^{+10\eta}\cap fA_j^{+10\eta}.
\]
If
\[
 C_i\neq fC_jf^{-1},
\]
the Periodicity Fact implies that the corresponding subinterval has
length at most $KL$.  Equivalently, if the corresponding subinterval
has length greater than $KL$, then
\[
 C_i=fC_jf^{-1}.
\]
In this case we say that $J_i$ and $fJ_j$ \emph{match}.  By Claim~3,
a given interval can match at most one interval on the other side.

We next show that every complete interval $J_i$ contained in the
corridor has a match.  The portion of $f\ell$ corresponding to $J_i$
has length $|J_i|$.  It must meet at least one interval $fJ_j$.
Indeed, otherwise it would be entirely contained in a single gap
between two consecutive intervals $fJ_j$, and hence
\[
 |J_i|\leq BL\leq\frac{s}{100},
\]
contrary to
\[
 |J_i|\geq\frac{s}{2}.
\]
Suppose therefore that it meets $q\geq1$ intervals $fJ_j$.  Apart
from the first and the last, all these intervals are completely
contained in this portion.  Since every $fJ_j$ has length at least
$s/2$, while $|J_i|\leq2s$, we have
\[
 (q-2)\frac{s}{2}\leq2s,
\]
and hence $q\leq6$.  Thus the portion corresponding to $J_i$ meets
at most six intervals $fJ_j$ and at most seven intervening or
terminal gaps.

If none of these intervals matched $J_i$, then each of the at most
six portions corresponding to an interval $fJ_j$ would have length
at most $KL$, by the Periodicity Fact.  On the other hand, each of
the at most seven portions corresponding to a gap has length at most
$BL$, by Claim~2.  Consequently,
\[
 |J_i|
 \leq6KL+7BL
 \leq\frac{13s}{100},
\]
contradicting
\[
 |J_i|\geq\frac{s}{2}.
\]
Hence every complete $J_i$ contained in the corridor has a unique
match $fJ_j$.

The matching preserves orientation.  Indeed, if $J_i$ matches
$fJ_j$, then
\[
 C_i=fC_jf^{-1},
\]
so the corresponding quasi-axes have the same pair of endpoints.
The matched portions have length greater than $KL$, while
$K\geq100\eta$ and $L\geq1$.  Two oppositely oriented
parametrizations of the same quasi-axis cannot remain within the
bounded distance prescribed by \eqref{FellowTravel} over such a
segment.  Hence the orientations induced by $\ell$ and $f\ell$
agree on matched intervals.  In particular, if $J_i$ and $J_{i'}$
match $fJ_j$ and $fJ_{j'}$, respectively, and $i<i'$, then
$j<j'$.

We claim moreover that consecutive intervals have consecutive
matches.  Suppose that
\[
 J_i\longleftrightarrow fJ_j,
 \qquad
 J_{i+1}\longleftrightarrow fJ_k.
\]
By orientation preservation, $k>j$.  Suppose that $k\geq j+2$.
Then the whole interval $fJ_{j+1}$ lies between the two matched
portions.  By Claim~3,
\[
 fC_{j+1}f^{-1}\neq C_i,C_{i+1}.
\]
The Periodicity Fact therefore bounds each of the portions of
$fJ_{j+1}$ corresponding to $J_i$ and $J_{i+1}$ by $KL$.  The
portion between them corresponds to the gap between $J_i$ and
$J_{i+1}$, whose length is at most $BL$ by Claim~2.  Hence
\[
 |fJ_{j+1}|
 \leq2KL+BL
 \leq\frac{3s}{100},
\]
contradicting
\[
 |fJ_{j+1}|=|J_{j+1}|\geq\frac{s}{2}.
\]
Thus
\begin{equation}\label{ConsecutiveMatches}
 J_i\longleftrightarrow fJ_j
 \quad\Longrightarrow\quad
 J_{i+1}\longleftrightarrow fJ_{j+1},
\end{equation}
whenever both $J_i$ and $J_{i+1}$ are complete intervals contained
in the corridor.

It remains only to find three consecutive complete intervals inside
the corridor.  Suppose, to the contrary, that $D$ contains at most
two complete intervals $J_i$.  Since $D$ is itself an interval, it
can then meet at most two further $J_i$'s, one partially at each end.
Thus it meets at most four intervals $J_i$ altogether.  Each has
length at most $2s$, and the portions of $D$ outside these intervals
are contained in at most five gaps, each of length at most
$BL\leq s/100$.  It follows that
\[
 |D|
 \leq4(2s)+5\frac{s}{100}
 <9s,
\]
contrary to $|D|>99s$.  Hence $D$ contains three consecutive complete
intervals
\[
 J_i,\qquad J_{i+1},\qquad J_{i+2}.
\]

Let $fJ_j$ be the unique match of $J_i$.  Applying
\eqref{ConsecutiveMatches} twice shows that the unique matches of
$J_{i+1}$ and $J_{i+2}$ are respectively $fJ_{j+1}$ and
$fJ_{j+2}$.  Therefore
\[
 fC_{j+r}f^{-1}=C_{i+r}
 \qquad(r=0,1,2),
\]
which is \eqref{ThreeMatches} and proves Claim~4.
\qed
\ \\ \\
\noindent\emph{Claim 5.} Equations \eqref{ThreeMatches} imply
$f\in\langle u\rangle$.

\smallskip
\noindent\emph{Proof of Claim 5.}
The first equality in \eqref{ThreeMatches} gives
\[
 fP_j=P_i c
\]
for some $c\in C$, since $C$ is self-normalizing.  Using the second
equality in \eqref{ThreeMatches}, we obtain
\[
 g^{-1}h^{-e_i}c h^{e_j}g\in C.
\]
But
\[
 h^{-e_i}c h^{e_j}\in C.
\]
Since $g\notin C$ and $C$ is malnormal, it follows that
\[
 h^{-e_i}c h^{e_j}=1.
\]
Consequently,
\[
 fP_{j+1}
 =fP_jh^{e_j}g
 =P_i c h^{e_j}g
 =P_i h^{e_i}g
 =P_{i+1}.
\]

The third equality in \eqref{ThreeMatches} now gives
\[
 g^{-1}h^{e_{j+1}-e_{i+1}}g\in C.
\]
Since $h^{e_{j+1}-e_{i+1}}\in C$, another application of
malnormality gives
\[
 h^{e_{j+1}-e_{i+1}}=1.
\]
As $h$ has infinite order,
\[
 e_{j+1}=e_{i+1}.
\]
The exponents are distinct modulo the period $m$, and hence
\[
 i-j=km
\]
for some $k\in\mathbb Z$.  By the periodic construction of the
prefixes,
\[
 P_{r+m}=uP_r,
\]
and therefore
\[
 P_{i+1}=P_{j+1+km}=u^kP_{j+1}.
\]
Together with $fP_{j+1}=P_{i+1}$, this yields
\[
 fP_{j+1}=u^kP_{j+1},
\]
and hence $f=u^k$.  This proves Claim~5.
\qed
\ \\ \\
Finally, Claims~4 and~5 show that if $f\notin\langle u\rangle$, then $\Delta(u,fuf^{-1})\leq100m^2L$. This proves the second assertion and completes the proof.
\end{proof}

\begin{remark}
The second assertion also implies that $u$ has no proper roots.
Indeed, suppose that $u=v^k$
for some $|k|>1$.  Since $v$ commutes with $u$, we have $\Delta(u,vuv^{-1})=\infty$.

The second assertion therefore gives $v\in\langle u\rangle$.  Thus
$v=u^n$ for some $n\in\mathbb Z$, and so $u=v^k=u^{nk}$, contradicting that $u$ has infinite order.
\end{remark}

\begin{corollary}\label{ArbitraryN}
Let $N$ be a positive integer, and choose an integer $m$ with
\[
 m>\max\{m_0,200N\}.
\]
For every pair $g,h\in\Gamma$ satisfying the hypotheses of
Proposition~\ref{HyperbolicSmallCancellation}, the element $u$ defined
there is $N$-small cancellation.
\end{corollary}

\begin{proof}
By Proposition~\ref{HyperbolicSmallCancellation},
\[
 \tau(u)\geq\frac12m^3L
\]
and, for every $f\notin\langle u\rangle$,
\[
 \Delta(u,fuf^{-1})\leq100m^2L.
\]
Since $L\geq1$,
\[
 \tau(u)\geq\frac12m^3>N.
\]
Moreover,
\[
 \frac{\Delta(u,fuf^{-1})}{\tau(u)}
 \leq\frac{200}{m}
 <\frac1N.
\]
Thus $u$ is $N$-small cancellation.
\end{proof}

In order to apply Corollary~\ref{ArbitraryN} to an arbitrary
noncommuting pair, it remains to put such a pair into the geometric
position required by Proposition~\ref{HyperbolicSmallCancellation}.  The
condition on the translation length causes no difficulty: in the
fixed Cayley graph $X$, the stable translation lengths of nontrivial
elements are uniformly bounded away from zero. The point is therefore to obtain, from an arbitrary noncommuting pair
$a,b$, a pair $g,h$, given by one of finitely many pairs of words,
which satisfies the hypotheses of
Proposition~\ref{HyperbolicSmallCancellation}.  The following proposition
provides such a pair uniformly in $a$ and $b$.

\begin{proposition}\label{prop:normalization}
There exists an integer $r=r(X)\geq1$ such that, for every
noncommuting pair $a,b\in\Gamma$, one of the three pairs
\[
 (b,a^r),\qquad
 (a,b^r),\qquad
 \bigl(a,(ab)^r\bigr),
\]
say $(g,h)$, satisfies
\[
 \tau(h)\geq1,\qquad [g,h]\neq1,
\]
and
\[
 d(o,g.o)\leq4\tau(h), \textrm{ for some } o\in A(h).
\]

\end{proposition}

\begin{proof}
Put
\[
 E=\inf_{z\in X}\max\{d(z,az),d(z,bz)\},
\]
the joint minimal displacement and choose $z\in X$ with $d(z,az),d(z,bz)\leq E+1$.
By Breuillard--Fujiwara \cite[Proposition~5.6]{BF}, there is an
absolute constant $K$ such that
\begin{equation}\label{eq:joint-displacement}
 E-K\delta\leq
 \max\{\tau(a),\tau(b),\tfrac12\tau(ab)\}.
\end{equation}

By \cite[Lemmas 2.13, 2.14 and 2.10]{DS20}, there is a constant
$D$ such that, if $\tau(v)>200\delta$, one can choose $o\in A(v)$ with
\begin{equation}\label{eq:axis}
 d(z,o)\leq\tfrac12d(z,vz)+D\delta.
\end{equation}
Indeed, 
the  minimum-displacement set lies within $90\delta$
of a geodesic with endpoints $v^-,v^+$, and its distance from $z$
is at most $\tfrac12d(z,vz)+5\delta$. Since the minimum displacement
is at least $\tau(v)$, we can apply Lemma 2.14. We divide the proof into two cases, according to whether $E$ is above or below a fixed threshold.

Take $E_0$ so that \[
 E_0-K\delta>200\delta,
 \qquad E_0\geq3+(2D+4K)\delta.
\]
First suppose $E>E_0$. Take $v\in\{a,b,ab\}$ realizing the
maximum in \eqref{eq:joint-displacement}, and take  $g$
so that
\[
 (g,v)\in\{(b,a),(a,b),(a,ab)\}.
\]
Then $\tau(v)\geq E-K\delta>200\delta$, while
$d(z,gz)\leq E+1$ and $d(z,vz)\leq2(E+1)$.
For $o$ from  \eqref{eq:axis},
\[
 d(o,go)\leq d(z,gz)+2d(z,o)
 \leq3(E+1)+2D\delta
 \leq4(E-K\delta)\leq4\tau(v).
\]
Thus this pair has the required properties, and so does
$(g,v^r)$ for every integer $r\geq1$, since
$A(v^r)=A(v)$ and $\tau(v^r)=r\tau(v)$.

Suppose now that $E\leq E_0$. Choose a vertex $x\in X$
with $d(x,z)\leq1$, and let $c\in\Gamma$ be its label. Then $d(c,ac),d(c,bc)\leq E_0+3$.
Hence, after simultaneous conjugation by $c^{-1}$, both $a$ and $b$
lie in the fixed finite ball of radius $E_0+3$ about $1$.
For each of the finitely many noncommuting pairs in this ball, choose
$o\in A(a)$. Since $\tau(a)>0$, we can choose one  $r\geq1$, large
enough for all these pairs, that satisfies
\[
 r\tau(a)\geq1,\qquad d(o,bo)\leq4r\tau(a).
\]
Conjugating back gives the required  properties for $(b,a^r)$. 
\end{proof}

We now have all the ingredients for the proof of the main theorem.

\begin{proof}[Proof of Theorem~\ref{thm:main-sc}]
Fix $N\geq1$, and let $r=r(X)$ be given by
Proposition~\ref{prop:normalization}. Choose
\[
 m>\max\{m_0,200N\}.
\]
For a noncommuting pair $a,b\in\Gamma$, Proposition~\ref{prop:normalization}
provides one of the three pairs
\[
 (b,a^r),\qquad (a,b^r),\qquad \bigl(a,(ab)^r\bigr),
\]
say $(g,h)$, satisfying the hypotheses of
Proposition~\ref{HyperbolicSmallCancellation}. Hence
\[
 h^{m^2+1}g\,h^{m^2+2}g\cdots h^{m^2+m}g
\]
is $N$-small cancellation.

Consequently, substituting the three possible choices of $(g,h)$ in the expression
above gives a finite set $W_N\subseteq\mathbb F(x,y)$, independent of
$a$ and $b$, one of whose elements is $N$-small cancellation at
$(a,b)$. 
\end{proof}

\section{Small cancellation and formal solutions}

The following theorem is a version of the implicit function theorem
\cite[Theorem 4]{MR2154989} (see also
\cite[Theorem 1.2]{MR2060243}). Originating in Merzlyakov's work on
systems of equations \cite{Merzlyakov1966}, the theorem was subsequently
extended to incorporate inequalities and became a fundamental ingredient
in the proof of the completeness of the first-order theory of nonabelian
free groups. Generalizations of Merzlyakov's theorem to
acylindrically hyperbolic groups, including
systems of equations and inequations, were
established by Andr\'e and Fruchter
\cite[Theorems~1.11 and~5.1]{AF22}.
The particular formulation needed here follows
from the same underlying techniques.

\begin{theorem}\label{ImplicitlLimitGroups}
Let $\Gamma$ be a non-elementary torsion-free hyperbolic group. Let $\Sigma(x,\bar y)=1$ and $\Psi(x,\bar y)\neq 1$ be, respectively,
a finite system of equations and a finite conjunction of inequalities,
all coefficient-free. Then there exists a positive integer $N$,
depending only on $\Sigma$, with the following property.
Suppose that $a\in\Gamma$ is $N$-small
cancellation, and that
\[
\Gamma\models
\exists\bar y\bigl(\Sigma(a,\bar y)=1\land
\Psi(a,\bar y)\neq 1\bigr).
\]
Then there exists a tuple of words $\bar w(x)$ with coefficients
in $\Gamma$ such that every word appearing in
$\Sigma(x,\bar w(x))$ is trivial in
$\Gamma*\mathbb F(x)$ and every word appearing in
$\Psi(x,\bar w(x))$ is nontrivial in
$\Gamma*\mathbb F(x)$.

Consequently,
\[
\Gamma*\mathbb F(z)\models
\exists\bar y\bigl(\Sigma(z,\bar y)=1\land
\Psi(z,\bar y)\neq 1\bigr).
\]
\end{theorem}



The preceding results yield the following finite test for coefficient-free
existential formulas.

\begin{lemma}\label{FiniteTest}
Let $\phi(x)$ be a coefficient-free existential formula. Then there exists
a finite nonempty family $\mathcal W_\phi\subseteq\mathbb F(p,q)$ such that for every pair of noncommuting elements
$a,b\in\Gamma$,
$$\Gamma\models
\bigwedge_{w\in\mathcal W_\phi}\phi(w(a,b))
\quad\Longrightarrow\quad
\Gamma*\langle z\rangle\models\phi(z).
$$
\end{lemma}
\begin{proof}
Write $\phi(x)$ as a finite disjunction
\[
\phi(x)=
\bigvee_{j=1}^r
\exists\bar y\,
\bigl(\Sigma_j(x,\bar y)=1\land
\Psi_j(x,\bar y)\neq1\bigr),
\]
where, by adjoining the equation $1=1$ if necessary, we may assume that
each $\Sigma_j$ is nonempty. Choose $N$ so that
Theorem~\ref{ImplicitlLimitGroups} applies to every disjunct, and let
$\mathcal W_\phi:=\mathcal W_N$ be given by
Theorem~\ref{thm:main-sc}.

Suppose that $\phi(w(a,b))$ holds for every
$w\in\mathcal W_\phi$. By Theorem~\ref{thm:main-sc}, some
$w(a,b)$ is $N$-small cancellation. Applying
Theorem~\ref{ImplicitlLimitGroups} to a disjunct of $\phi(w(a,b))$
which holds, we obtain
\[
\Gamma*\langle z\rangle\models\phi(z),
\]
as required.
\end{proof}

\section{Nonexistence of model companions}

We now apply the uniform small cancellation theorem to the model companion problem for torsion-free hyperbolic groups.

\begin{proof}[Proof of Theorem~\ref{thm:model-companion}]
Suppose that $\mathcal U$ is a model companion of $Th(\Gamma)$. By Fact \ref{ModelComplete}, there is an existential formula $\phi(x)=\exists\bar y\,\psi(x,\bar y)$, with $\psi$ quantifier-free, such that
\begin{equation}\label{Nonsquares}
 \mathcal U\models\forall x\,
 \bigl(\phi(x)\leftrightarrow\neg\exists y\,(y^2=x)\bigr).
\end{equation}
Let
\[
 \theta(x):=\bigl(\exists y\,(y^2=x)\bigr)\vee\phi(x).
\]
This is an existential formula, and $\mathcal U\models\forall x\,\theta(x)$. Let $\mathcal W_\theta$ be the finite family supplied by Lemma \ref{FiniteTest}.

The theory $\mathcal U$ has a nonabelian model $G$. Otherwise the universal sentence $\forall p\forall q\,[p,q]=1$ would belong to $\mathcal U^{\forall}=Th(\Gamma)^{\forall}$, a contradiction. Choose noncommuting $a,b\in G$. Then $G$ satisfies the sentence
\begin{equation}\label{FiniteSentence}
 \exists p\exists q\left([p,q]\neq1\ \wedge\
  \bigwedge_{w\in\mathcal W_\theta}\theta(w(p,q))\right).
\end{equation}
Since $\mathcal W_\theta$ is finite, this sentence is existential. By the completeness of $Th(\Gamma)$ (see the observation at the end of Section~2) we get that \eqref{FiniteSentence} holds in $\Gamma$. Applying Lemma \ref{FiniteTest} to its witnesses gives
\begin{equation}\label{FreeGenerator}
 \Gamma*\mathbb{F}(z)\models\theta(z).
\end{equation}

The element $z$ has no square root in $\Gamma*\mathbb{F}(z)$. Thus 
$\Gamma*\mathbb{F}(z)\models\phi(z)$. By the observation in the last paragraph of Subsection~\ref{sub:model-theory} we get
$
\Gamma*\mathbb{F}(z)\models\phi(z^2).
$
On the other hand, \eqref{Nonsquares} implies the universal sentence
\[
\forall x\forall\bar y\forall t\,
\neg\bigl(\psi(x,\bar y)\land t^2=x\bigr),
\]
which belongs to $U^\forall=Th(\Gamma)^\forall$ and hence holds in
$\Gamma*\mathbb{F}(z)$. This is a contradiction, since $z$ is
a square root of $z^2$.
\end{proof}

Finally, by \cite{KM12}, every nonabelian $\Gamma$-limit group is
universally equivalent to $\Gamma$. Hence the main theorem also yields:

\begin{corollary}
For any non-elementary torsion-free hyperbolic group $\Gamma$, the first-order theory of any nonabelian $\Gamma$-limit group does not admit a model companion.
\end{corollary}

\ \\
{\bf Olga Kharlampovich}\\
Department of Mathematics and Statistics\\
Hunter College, City University of New York (CUNY)\\
695 Park Avenue\\
New York, NY 10065\\
U.S.A.
\ \\ \\
{\bf Rizos Sklinos} (corresponding author: rizos.sklinos@amss.ac.cn)\\
Academy of Mathematics \& Systems Science\\
Chinese Academy of Sciences\\
Room 1016 SiYuan Building\\
No.55 of Zhongguancun East Road\\
Haidian District,  100190\\
P. R. China

\end{document}